\documentclass[12pt]{amsart}

\usepackage[margin=0.8in]{geometry}
\usepackage{amsmath, amssymb, amsthm, mathtools}
\usepackage{microtype}
\usepackage{hyperref}

\newtheorem{theorem}{Theorem}[section]
\newtheorem{proposition}[theorem]{Proposition}
\newtheorem{corollary}[theorem]{Corollary}
\theoremstyle{definition}
\newtheorem{definition}[theorem]{Definition}
\newtheorem{example}[theorem]{Example}
\theoremstyle{remark}

\newtheorem{question}[theorem]{Question}

\newcommand{\N}{\mathbb{N}}
\newcommand{\Z}{\mathbb{Z}}
\newcommand{\Sset}{\mathfrak{S}}
\newcommand{\XD}{\mathrm{XDes}}
\newcommand{\Des}{\mathrm{Des}}

\newcommand{\set}[1]{\left\{#1\right\}}
\newcommand{\abs}[1]{\left|#1\right|}
\newcommand{\brac}[1]{\left[#1\right]}

\begin{document}

\title{Permutations with a given $X$-descent set}
\author{Mohamed Omar}
\address{Department of Mathematics \& Statistics. York University. 4700 Keele St. Toronto, Canada. M3J 1P3}
\email{omarmo@yorku.ca}


\begin{abstract}
Building on the work of Grinberg and Stanley, we begin a systematic study of permutations with a prescribed $X$-descent set. In particular, for a set $X \subseteq \mathbb{N}^2$, and $I \subseteq [n-1]$, we study the permutations $\pi \in \mathfrak{S}_n$ whose $X$-descent set is precisely $I$, meaning $(\pi_i,\pi_{i+1}) \in X$ precisely when $i \in I$. The central focus is enumerating these permutations for a fixed $X,I$ and $n$: this count is denoted by $d_X(I;n)$. We derive a recursion which under expected conditions simplifies to a binomial-type recurrence determined entirely by the values $d_X(\emptyset;n)$.
This extends the work of D\'iaz-Lopez et al.\ on descent polynomials.
The resulting reduction shows that the general statistic $d_X(I;n)$ is typically governed by the ``descent-free'' quantities $d_X(\emptyset;n)$, motivating a closer analysis of these numbers.
We observe that $d_X(\emptyset;n)$ enumerates Hamiltonian paths in a directed graph canonically associated to $X$. We then record several families of sets $X$ for which $d_X(\emptyset;n)$ is explicit or effectively computable. This includes families with periodicity for which transfer matrix methods apply, and families with succession-type relations where inclusion-exclusion applies. We then investigate the typical behavior of $d_X(\emptyset;n)$ from a probabilistic perspective.
\end{abstract}


\maketitle
\date{\today}

\section{Introduction}

\noindent A ubiquitous statistic in the permutation patterns literature is the descent set of a permutation $\pi=\pi_1\pi_2\cdots\pi_n\in\Sset_n$, which is given by
\[
  \Des(\pi)\coloneqq \set{i\in\brac{n-1}:\ \pi_i>\pi_{i+1}}.
\]
For a fixed set $I\subseteq[n-1]$, let $d(I;n)$ denote the number of permutations in $\Sset_n$ with descent set $I$.
MacMahon proved that, for each fixed $I$, the function $n\mapsto d(I;n)$ is (eventually) polynomial in $n$
\cite{macmahon1915combinatory}.
These polynomials, their binomial basis expansions, and associated recurrences, were studied in detail by
D\'{\i}az-Lopez et al. \cite{descentpolys}. In this article, we study a direct analogue of this function for a generalization of descent sets. Given a set $X\subseteq\N^2$ and  $\pi=\pi_1\pi_2\cdots\pi_n\in\Sset_n$, the \emph{$X$-descent set} of $\pi$ as defined by Grinberg and Stanley \cite{grinberg-stanley-redei-berge} is
\[
  \XD(\pi)\coloneqq \set{i\in\brac{n-1}:\,(\pi_i,\pi_{i+1})\in X}.
\]
In the spirit that inspired the work of D\'iaz-Lopez et al. \cite{descentpolys} and Billey et al. \cite{billey-burdzy-sagan-peaks}, this motivates the central definition of our article.

\begin{definition}
For $I \subseteq [n-1]$, define $\mathcal{D}_X(I;n)$ to be the set of permutations in $\mathfrak{S}_n$ whose $X$-descent set is precisely $I$. In other words,
\[
  \mathcal{D}_X(I;n)\coloneqq \set{\pi\in\Sset_n:\ \XD(\pi)=I}, 
\]
Its size is denoted by $d_X(I;n)\coloneqq \abs{\mathcal{D}_X(I;n)}.$
\end{definition}

\medskip 
\noindent Several of our arguments are more transparent if we allow permutations of an arbitrary finite set of labels, so we have the following related definition.

\begin{definition} Let $S\subseteq\N$ be finite with $\abs{S}=n$. Write $\Sset_S$ for the set of bijections $\pi:[n]\to S$, written in
one-line notation $\pi=\pi_1\cdots\pi_n$. For $I\subseteq[n-1]$, define
\[
  \mathcal{D}_X(I;S)\coloneqq\set{\pi\in\Sset_S:\ \XD(\pi)=I},\qquad d_X(I;S)\coloneqq \abs{\mathcal{D}_X(I;S)}.
\]
When $S=[n]$ we recover $\mathcal{D}_X(I;n)$ and $d_X(I;n)$.
\end{definition}

The notion of an $X$-descent set appears throughout related literature. When $X=\{(i,j) \ : \ i>j\}$ we have  $\mathcal{D}_X(I;n)=\mathcal{D}(I;n)$, the set of permutations in $\mathfrak{S}_n$ with descent set $I$. If $I \subseteq [n-1]$ and $\mathcal{P}(I;n)$ is the set of permutations in $\mathfrak{S}_n$ with peak set $I$, then we can express $\mathcal{P}(I;n)$ in terms of descent sets: again letting $X=\{(i,j) \ : \ i>j\}$, we have
\[
P(I;n) = \bigsqcup_{\substack{J \subseteq [n-1] \\ i \in [2,n-1] \colon i-1 \notin J, \ i \in J}} D_X(J;n).
\]
This follows since a peak at an index is an ascent directly before the index followed by a descent at that index. Many other classes of permutations can be cast as permutations with a given $X$-descent set as well. For instance, reverse succession-free permutations are those permutations $\pi$ for which $\pi_j - \pi_{j+1} \neq 1$ for any $j$, and so permutations that are reverse succession-free at prescribed indices $I$ are precisely those in the set $\mathcal{D}_X(I;n)$ where $X=\{(i+1,i) \ : \ i \in \mathbb{N}\}$ (we will see more on this later).

Our first main goal in this article is to give a recursion for $d_X(I;S)$ that is valid for every relation $X$ and every finite label set $S$.
Under a natural standardization-invariance hypothesis, this recursion collapses to a binomial recursion, recovering the
MacMahon-type polynomiality familiar from descent sets and connecting back to the insertion recurrences in
\cite{descentpolys}. This binomial recurrence hinges on the values $d_X(\emptyset;n)$ and so our second goal is to thoroughly study the special case $I=\emptyset$: we note that $d_X(\emptyset;n)$ counts Hamiltonian
paths in a directed graph determined by $X$.
This interpretation makes it clear why $d_X(\emptyset;n)$ can be difficult to compute in general, but it also brings
in standard tools from graph enumeration and probabilistic combinatorics. This affords us the capability to enumerate $d_X(\emptyset;n)$ for many families of set $X$.

\medskip
The paper is organized as follows.
Section~\ref{sec:recurrence} proves the basic recurrences for $d_X(I;S)$ and discusses the standardization-invariant
setting, including an insertion recursion in the spirit of \cite{descentpolys}. Because Section~\ref{sec:recurrence} reduces the enumeration of $d_X(I;n)$ to the numbers $d_X(\emptyset;n)$, the remaining focus of the paper is enumerating this statistic. 
Section~\ref{sec:hamps} interprets $d_X(\emptyset;n)$ as a Hamiltonian-path count and for certain sets $X$, writes $d_X(\emptyset;n)$ in terms of  formulae of
Grinberg and Stanley.
The rest of Section~\ref{sec:compute} computes $d_X(\emptyset;n)$ for several accessible families, illustrating how combinatorial methods in the literature apply appropriately. This includes applying transfer matrix methods when $X$ has certain periodicity, and standard inclusion-exclusion when $X$ consists of successions. The section ends by establishing a probabilistic argument showing that $d_X(\emptyset;n)$ is typically much larger than polynomial in $n$. We close in Section~\ref{sec:open} with a list of open problems.

\section{Recurrence formulae and polynomiality}\label{sec:recurrence}
\subsection{Base Recurrence}

The articles of MacMahon \cite{macmahon1915combinatory} and D\'{\i}az-Lopez et al. \cite{descentpolys} prove that $d(I;n)$ is polynomial in $n$ by developing a recursive formula for $d(I;n)$ in terms of $d(I';n)$ with $|I'|<|I|$. A natural question is whether $d_X(I;n)$ also admits such a recursive formula that can be used to establish that it is polynomial in $n$. The binomial recursion for descent polynomials is often presented as a recursion that depends only on the size of a
prefix set and a suffix set. For a general relation $X$, this is false: the behavior depends on \emph{which labels} appear in the
prefix. The correct recursion is therefore a sum over label subsets.

\begin{theorem}\label{thm:recursion}
Let $X\subseteq\N^2$. Let $S\subseteq\N$ with $\abs{S}=n$, and let $I\subseteq[n-1]$ be nonempty. 
Write $m=\max(I)$ and $I^-=I\setminus\set{m}$.
Then
\begin{equation}\label{eq:recursion-subset-sum}
d_X(I;S)
=
\sum_{\substack{A\subseteq S\\ \abs{A}=m}}
d_X(I^-;A)\,d_X(\emptyset;S\setminus A)
-
d_X(I^-;S).
\end{equation}
\end{theorem}

\begin{proof}
Consider the set
\[
P\coloneqq\Bigl\{\pi=\pi_1\cdots\pi_n\in\Sset_S:\ \XD(\pi_1\cdots\pi_m)=I^-\ \text{and}\ \XD(\pi_{m+1}\cdots\pi_n)=\emptyset\Bigr\}.
\]
For each \(A\subseteq S\) with \(\abs{A}=m\), a permutation \(\pi\in P\) with \(\{\pi_1,\dots,\pi_m\}=A\) consists of a
prefix in \(\mathcal{D}_X(I^-;A)\) together with a suffix in \(\mathcal{D}_X(\emptyset;S\setminus A)\), and these two
choices are independent.  Summing over the choice of \(A\) gives
\[
 \abs{P}=\sum_{\substack{A\subseteq S\\ \abs{A}=m}} d_X(I^-;A)\,d_X(\emptyset;S\setminus A).
\]

On the other hand, if \(\pi\in P\), then all potential \(X\)-descents are prescribed except possibly at the boundary
position \(m\).  The condition \(\XD(\pi_1\cdots\pi_m)=I^-\) forces the descent pattern on the first \(m-1\) positions,
and \(\XD(\pi_{m+1}\cdots\pi_n)=\emptyset\) forces the last \(n-m-1\) positions to be non-descents.  Thus
\(\XD(\pi)\) is either \(I^-\) or \(I\), according as \((\pi_m,\pi_{m+1})\notin X\) or \((\pi_m,\pi_{m+1})\in X\).
Equivalently, $P=\mathcal{D}_X(I^-;S)\,\dot\cup\,\mathcal{D}_X(I;S)$, so
$\abs{P}=d_X(I^-;S)+d_X(I;S).$ Equating the two expressions for \(\abs{P}\) yields~\eqref{eq:recursion-subset-sum}.
\end{proof}

\subsection{Standardization-Invariance}
If $X$ is \emph{standardization-invariant}, meaning that for every finite set $S$ and every $I$ the number $d_X(I;S)$
depends only on $\abs{S}$ (equivalently, $d_X(I;S)=d_X(I;\abs{S})$ after order-preserving relabeling), then
$d_X(I^-;A)=d_X(I^-;m)$ and $d_X(\emptyset;S\setminus A)=d_X(\emptyset;n-m)$ for every $A$ with $\abs{A}=m$.
In this case~\eqref{eq:recursion-subset-sum} collapses to
\begin{equation}\label{eq:recursion-binomial}
d_X(I;n)=\binom{n}{m}\,d_X(I^-;m)\,d_X(\emptyset;n-m)-d_X(I^-;n).
\end{equation}
For ordinary descents this is
essentially MacMahon's recursion \cite{macmahon1915combinatory}. 

Moreover, notice that for a fixed relation $X$ and fixed $I$, the recursion of Theorem~\ref{thm:recursion} can be used as a computational
tool, but it exposes an inherent combinatorial explosion.
Even when $S=[n]$, the right-hand side of~\eqref{eq:recursion-subset-sum} sums over all $\binom{n}{m}$ subsets $A$ of
size $m=\max(I)$ and requires values of $d_X(I^-;A)$ that depend on the \emph{actual} labels in $A$.
In the worst case this forces one to keep track of $d_X(J;A)$ for exponentially many label sets $A$,.
This is another way to see why standardization invariance is the natural dividing line between ``polynomial'' and
``wild'' behavior: once $d_X(J;A)$ depends only on $\abs{A}$, the recursion collapses to a scalar recursion in $n$.

\begin{theorem}\label{thm:recursive}
Assume that $X$ is standardization-invariant.
Let $I=\{i_1<i_2<\cdots<i_k\}\subseteq[n-1]$, set $i_0=0$, and let $n\ge \max(I)+1$.
Then
\begin{equation}\label{eq:IE-expanded}
d_X(I;n)
=
\sum_{\{j_1<\cdots<j_r\}\subseteq\{1,\dots,k\}}
(-1)^{k-r}
\binom{n}{n-i_{j_r},i_{j_r}-i_{j_{r-1}},\ldots,i_{j_2}-i_{j_1}}
\prod_{t=0}^{r} d_X\!\left(\emptyset;\,i_{j_{t+1}}-i_{j_t}\right),
\end{equation}
where we interpret $j_0=0$ and $j_{r+1}=k+1$ with $i_{k+1}=n$.
(When $r=0$ the empty multinomial coefficient is $1$ and the last product is $d_X(\emptyset;n)$.)
\end{theorem}

\begin{proof}
We prove~\eqref{eq:IE-expanded} by induction on \(k=\abs{I}\), iterating the binomial recursion~\eqref{eq:recursion-binomial}.
If \(k=0\), then \(I=\emptyset\) and the sum in~\eqref{eq:IE-expanded} has only the empty choice \(r=0\), which gives
\(d_X(\emptyset;n)\).
If \(k=1\) and \(I=\{m\}\), the two subsets \(\emptyset\) and \(\{1\}\) recover~\eqref{eq:recursion-binomial} with
\(I^-=\emptyset\).

Now let \(k\ge 2\) and assume the formula holds for all sets of size \(k-1\).
We have \(I=\{i_1<i_2<\cdots<i_k\}\subseteq[n-1]\), with \(m=i_k\), and \(I^-=I\setminus\{m\}\).
Standardization invariance gives
\begin{equation}\label{eq:IE-induction-start}
 d_X(I;n)=\binom{n}{m}\,d_X(I^-;m)\,d_X(\emptyset;n-m)\ -\ d_X(I^-;n).
\end{equation}

Apply the induction hypothesis to \(d_X(I^-;m)\) and \(d_X(I^-;n)\).
For \(d_X(I^-;m)\), $I^-=\{i_1<i_2<\cdots<i_{k-1}\}$ and the ambient set has size $m$, so
\[
d_X(I^-;m)
=
\sum_{\substack{\{j_1<\cdots<j_r\}\subseteq [k-1]\\ 0\le r\le k-1}}
(-1)^{k-1-r}\,
\binom{m}{i_{j_r}}\binom{i_{j_r}}{i_{j_{r-1}}}\cdots\binom{i_{j_2}}{i_{j_1}}\,
\prod_{t=0}^{r} d_X(\emptyset;\,i_{j_{t+1}}-i_{j_t}),
\]
where we use the same conventions as in~\eqref{eq:IE-expanded} and interpret \(j_0=0\), \(i_{j_0}=0\), and
\(i_{j_{r+1}}=m\).
Similarly,
\[
d_X(I^-;n)
=
\sum_{\substack{\{j_1<\cdots<j_r\}\subseteq [k-1]\\ 0\le r\le k-1}}
(-1)^{k-1-r}\,
\binom{n}{n-i_{j_r},i_{j_r}-i_{j_{r-1}},\ldots,i_{j_2}-i_{j_1}}\,
\prod_{t=0}^{r} d_X(\emptyset;\,i_{j_{t+1}}-i_{j_t}),
\]
now with \(i_{j_{r+1}}=n\).

Substituting these two expansions into~\eqref{eq:IE-induction-start} produces two sums indexed by subsets of \([k-1]\).
In the first sum, the prefactor \(\binom{n}{m}\,d_X(\emptyset;n-m)\) simply appends \(m=i_k\) as the final selected
descent position and appends the final gap factor \(d_X(\emptyset;n-m)\).  Reindexing by adjoining \(k\) to the subset
\(\{j_1<\cdots<j_r\}\subseteq[k-1]\) therefore yields precisely the terms in~\eqref{eq:IE-expanded} indexed by subsets of
\([k]\) that contain \(k\), with the correct sign \((-1)^{k-(r+1)}=(-1)^{k-1-r}\).
The second sum contributes the terms indexed by subsets of \([k]\) that do not contain \(k\), and the extra minus sign
changes \((-1)^{k-1-r}\) into \((-1)^{k-r}\).
Together these two families of terms give exactly the full sum~\eqref{eq:IE-expanded}.
\end{proof}

When $X$ induces the classical descent set, the factors $d_X(\emptyset;\,i_{j_{t+1}}-i_{j_t})$ in~\eqref{eq:IE-expanded}
are all equal to $1$ (there is a unique increasing permutation of a given length), and the formula reduces to
MacMahon's classical inclusion-exclusion expression for $d(I;n)$ in terms of binomial coefficients
\cite{macmahon1915combinatory}.  In general, the same binomial ``skeleton'' persists, but the block factors
$d_X(\emptyset;\cdot)$ measure how many ways the gaps between descent positions can be filled without creating
additional $X$-descents.  One can view~\eqref{eq:IE-expanded} as an inclusion-exclusion over which of the designated
``candidate descents'' actually occur, together with a multinomial choice of which labels land in each block.
This viewpoint is closely related to the classical inclusion-exclusion expansions for descent-set numbers and to refinements by compositions; see \cite{stanley-ec1}.

\medskip
A different recursion, which avoids cancellation, arises by inserting the largest label $n+1$.
The classical version for ordinary descents is due to D\'{\i}az-Lopez et al.
\cite{descentpolys}. For $X$-descents one needs a hypothesis ensuring that $n+1$ behaves uniformly with respect to $X$.

\begin{proposition}\label{prop:cancellationfree}
Fix $n\ge 1$ and assume that
\begin{equation}\label{eq:maximal-assumption}
(n{+}1,i)\in X\ \text{ and }\ (i,n{+}1)\notin X\qquad\text{for every }i\in[n].
\end{equation}
Let $I=\{i_1<i_2<\cdots<i_\ell\}\subseteq[n]$ and define
\[
I' \coloneqq \set{i_k:\ i_k-1\notin I},\qquad
I''\coloneqq I'\setminus\set{1}.
\]
For each $k\in[\ell]$ set
\[
I_k\coloneqq \set{i_1,\dots,i_{k-1},\,i_k-1,\,i_{k+1}-1,\dots,i_\ell-1}\setminus\set{0},
\]
\[
\widehat{I}_k\coloneqq \set{i_1,\dots,i_{k-1},\,i_{k+1}-1,\dots,i_\ell-1}\setminus\set{0}.
\]
Then
\begin{equation}\label{eq:insertion-recursion}
d_X(I;n{+}1)
=
d_X(I;n)
+
\sum_{i_k\in I''} d_X(I_k;n)
+
\sum_{i_k\in I'} d_X(\widehat{I}_k;n).
\end{equation}
\end{proposition}

\begin{proof}
Let \(\pi=\pi_1\cdots\pi_{n+1}\in\mathcal{D}_X(I;n{+}1)\), and let \(t\) be the position of \(n{+}1\) in \(\pi\).
Delete \(n{+}1\) to obtain a permutation \(\pi'\in\Sset_n\). We are given that every pair of the form \((n{+}1,i)\) is an \(X\)-descent
and every pair of the form \((i,n{+}1)\) is not an \(X\)-descent.
Consequently, the only descent information lost when passing from \(\pi\) to \(\pi'\) comes from the two elements in positions before and after \(n{+}1\); all other pairs are unchanged (up to the index shift after position \(t\)).

If \(t=n{+}1\), then \(n{+}1\) is appended at the end, so $\XD(\pi')=\XD(\pi)=I$.  This accounts for the term \(d_X(I;n)\). Now assume \(t\le n\), so \(n{+}1\) has a successor.  

Since \((n{+}1,\pi_{t+1})\in X\), we have \(t\in I\).
If \(t>1\), then \((\pi_{t-1},n{+}1)\notin X\), so \(t-1\notin I\).  Thus \(t\) is a descent position whose predecessor
is not, i.e.\ \(t\in I' \), and in fact \(t\in I''\) when \(t>1\).
Write \(t=i_k\) accordingly. Deleting \(n{+}1\) removes the instances of \((\pi_{t-1},n{+}1)\) and \((n{+}1,\pi_{t+1})\) and replaces them by the instance \((\pi_{t-1},\pi_{t+1})\), which occurs at position \(t-1\) in \(\pi'\).
All descent positions after $t$ in \(\pi\) shift down by one in \(\pi'\), and the descent at \(t\) disappears because the
instance \((n{+}1,\pi_{t+1})\) is removed.
Therefore \(\XD(\pi')\) is determined by whether the pair $(\pi_{t-1},\pi_{t+1})$ lies in \(X\):
\[
(\pi_{t-1},\pi_{t+1})\in X \Longleftrightarrow t-1\in\XD(\pi') \Longleftrightarrow \XD(\pi')=I_k,
\]
and
\[
(\pi_{t-1},\pi_{t+1})\notin X \Longleftrightarrow t-1\notin\XD(\pi') \Longleftrightarrow \XD(\pi')=\widehat{I}_k.
\]

Conversely, start with \(\pi'\in\mathcal{D}_X(I_k;n)\) and insert \(n{+}1\) at position \(t=i_k\).
This destroys the \(X\)-descent at position \(t-1\) (which is present in \(\pi'\) by definition of \(I_k\)) because
\((\pi'_{t-1},n{+}1)\notin X\), and it creates an \(X\)-descent at position \(t\) because \((n{+}1,\pi'_t)\in X\).
All later descent positions shift up by one.  The result is a permutation in \(\mathcal{D}_X(I;n{+}1)\), and deletion at
position \(t\) recovers \(\pi'\).  The same insertion map sends \(\mathcal{D}_X(\widehat{I}_k;n)\) bijectively onto the
subfamily where the bridged pair is not in \(X\).

Finally, when \(t=1\) we necessarily have \(1\in I\), and deleting the initial \(n{+}1\) simply shifts every descent
position down by one.  This corresponds to the term with \(i_k=1\in I'\), namely \(d_X(\widehat{I}_k;n)\).

We have therefore partitioned \(\mathcal{D}_X(I;n{+}1)\) into disjoint subclasses according to the position of \(n{+}1\)
and the bijections above identify these subclasses with
\(\mathcal{D}_X(I;n)\), \(\mathcal{D}_X(I_k;n)\) for \(i_k\in I''\), and \(\mathcal{D}_X(\widehat{I}_k;n)\) for \(i_k\in I'\). Taking cardinalities gives~\eqref{eq:insertion-recursion}.
\end{proof}

\subsection{Periodic relations.}
A second natural invariance is \emph{periodicity modulo $m$}. In this setting, counting permutations reduces to
counting residue words with prescribed adjacent transitions, a standard finite-state enumeration problem.

\begin{theorem}\label{thm:periodic-reduction}
Let $m\ge 1$ and suppose $X$ is periodic modulo $m$, i.e.\ there exists
$f:(\Z/m\Z)^2\to\{0,1\}$ such that for all distinct $a,b\in\N$,
\[
(a,b)\in X\quad\Longleftrightarrow\quad f(a\bmod m,\ b\bmod m)=1.
\]
Fix $n\ge 1$ and $I\subseteq[n-1]$. For each residue class $r\in\Z/m\Z$ let
\[
\ell_r(n)\coloneqq \abs{\set{t\in[n]: t\equiv r\pmod m}}.
\]
Let $\mathcal{W}_{f,I}(n)$ be the set of words $w=w_1\cdots w_n$ over $\Z/m\Z$ with content
$\abs{\{i:\,w_i=r\}}=\ell_r(n)$ for all $r$, and such that
\[
f(w_i,w_{i+1})=1\ \Longleftrightarrow\ i\in I\qquad\text{for all }i\in[n-1].
\]
Then
\[
d_X(I;n)=\abs{\mathcal{W}_{f,I}(n)}\cdot \prod_{r\in\Z/m\Z}\ell_r(n)!.
\]
\end{theorem}

\begin{proof}
Consider the map
\[
 w:\Sset_n\longrightarrow (\Z/m\Z)^n,\qquad
 w(\pi)\coloneqq (\pi_1\bmod m,\ldots,\pi_n\bmod m),
\]
For fixed \(n\), the content of \(w(\pi)\) is forced: among the
labels \(1,2,\dots,n\) the residue class \(r\in\Z/m\Z\) occurs exactly \(\ell_r(n)\) times by definition. Now,
for \(\pi\in\Sset_n\) and \(1\le i\le n-1\),
\[
 i\in\XD(\pi)
 \Longleftrightarrow
 (\pi_i,\pi_{i+1})\in X
 \Longleftrightarrow
 f(\pi_i\bmod m,\ \pi_{i+1}\bmod m)=1
 \Longleftrightarrow
 f(w_i(\pi),w_{i+1}(\pi))=1.
\]
Therefore \(\XD(\pi)=I\) holds exactly when \(w(\pi)\) satisfies the
 rule defining \(\mathcal{W}_{f,I}(n)\).  In other words, \(w\) maps \(\mathcal{D}_X(I;n)\) onto
\(\mathcal{W}_{f,I}(n)\). It remains then to count the size of the preimage of any element in \(\mathcal{W}_{f,I}(n)\).  

Fix \(w'=w_1\cdots w_n\in\mathcal{W}_{f,I}(n)\).  For each residue class
\(r\in\Z/m\Z\), the set of labels in \([n]\) congruent to \(r\) modulo \(m\) has size \(\ell_r(n)\), and the word \(w'\)
specifies exactly \(\ell_r(n)\) positions in which residue \(r\) must appear.  To construct a permutation \(\pi\) with
\(w(\pi)=w'\), we must bijectively assign to those positions the \(\ell_r(n)\) distinct integers in \([n]\) that are
congruent to \(r\), in some order.  There are \(\ell_r(n)!\) choices for this assignment, independently for each residue
class \(r\), so every \(w' \in\mathcal{W}_{f,I}(n)\) has exactly \(\prod_{r\in\Z/m\Z}\ell_r(n)!\) preimages under the map
\(w\). Since the size of the preimage of $w'$ is independent of the choice of $w'$, we conclude that
\[
 d_X(I;n)=\abs{\mathcal{W}_{f,I}(n)}\cdot \prod_{r\in\Z/m\Z}\ell_r(n)!,
\]
which is the desired formula.  
\end{proof}

Theorem~\ref{thm:periodic-reduction} reduces the permutation count to a purely finite-state word problem.
Fix $m$ and let $A_1$ (resp.\ $A_0$) be the $m\times m$ $0$-$1$ adjacency matrix of the directed graph of residue
transitions allowed when $i\in I$ (resp.\ when $i\notin I$), i.e.,
\[
(A_1)_{rs}=\mathbf{1}_{f(r,s)=1},\qquad (A_0)_{rs}=\mathbf{1}_{f(r,s)=0}.
\]
If one \emph{does not} fix content, then the number of residue words of length $n$ satisfying the transition rule
``use $A_1$ at steps in $I$ and $A_0$ at steps outside $I$'' is 
$\mathbf{1}^\top A_{\varepsilon_1}A_{\varepsilon_2}\cdots A_{\varepsilon_{n-1}}\mathbf{1}$, where
$\varepsilon_i=\mathbf{1}_{i\in I}$.
This is a standard transfer-matrix computation for regular languages \cite{flajolet-sedgewick}.

Fixing the content vector $(\ell_r(n))_{r\in\Z/m\Z}$ is the only genuinely new ingredient.
One convenient way to incorporate content is to weight each letter $r$ by a variable $y_r$ (we will expand on this in
Theorem~\ref{thm:quasipolynomial}), but now the transition matrices, call them $M_1$ and $M_0$ respectively, look like:
\[
M_1(y)_{rs}=\mathbf{1}_{f(r,s)=1} \cdot y_s,\qquad M_0(y)_{rs}=\mathbf{1}_{f(r,s)=0} \cdot y_s.
\]
Then the multivariate generating function for residue words satisfying the $I$-dependent rule is
\[
x\,\mathbf{y}^\top\Bigl(M_{\varepsilon_1}(y)\,M_{\varepsilon_2}(y)\cdots M_{\varepsilon_{n-1}}(y)\Bigr)\mathbf{1},
\]
where $\mathbf{y}^T=(y_0,\ldots,y_{m-1})$. Extracting the coefficient of $y_0^{\ell_0(n)}\cdots y_{m-1}^{\ell_{m-1}(n)}$ produces $\abs{\mathcal{W}_{f,I}(n)}$.
For fixed $m$ and fixed $I$, this is computable, though it
typically does not simplify to a ``closed form.''  

Finally, if the set of required descent positions $I$ itself has a regular description (for example, if the indicator
$\mathbf{1}_{i\in I}$ is eventually periodic in $i$), then one can often sum over $n$ and obtain a genuinely rational
(or at least $D$-finite) generating function in $x$ for the family $\{d_X(I;n)\}_{n\ge 0}$ by standard 
techniques \cite{flajolet-sedgewick}.  We do not pursue this direction here, but it suggests that periodicity in $X$
can be fruitfully combined with periodicity in the \emph{descent pattern} itself.

\begin{example} Let $m=2$ and let $X$ consist of all pairs $(a,b)$ with $a$ even and $b$ odd. We compute $d_X(\emptyset;n)$ using Theorem~\ref{thm:periodic-reduction}. First, we have $f:(\mathbb{Z}/2\mathbb{Z})^2 \to \{0,1\}$ is given by $f(0,1)=1$ and $f(r,s)=0$ otherwise. Now we determine $\mathcal{W}_{f,\emptyset}(n)$. If $w=w_1w_2 \cdots w_n$ is a word over $\mathbb{Z}/2\mathbb{Z}$, the requirement $f(w_i,w_{i+1})=1$ if and only if $i \in I$ implies $f(w_i,w_{i+1})=0$ for all $i \in [n-1]$. Therefore $\mathcal{W}_{f,\emptyset}(n)$ consists only of the unique length $n$ word over $\mathbb{Z}/2\mathbb{Z}$ starting with $\ell_1(n)$ ones and ending with $\ell_0(n)$ zeros. So, $\abs{\mathcal{W}_{f,\emptyset}(n)}=1$, and therefore Theorem~\ref{thm:periodic-reduction} gives the explicit formula
\[
d_X(\emptyset;n)=\ell_0(n)!\,\ell_1(n)! = \Bigl\lfloor\tfrac{n}{2}\Bigr\rfloor!\,\Bigl\lceil\tfrac{n}{2}\Bigr\rceil!.
\]
In particular, even with a very simple periodic rule, $d_X(\emptyset;n)$ can grow on the order of
$\bigl(\tfrac{n}{2}\bigr)!^2$, illustrating again that polynomiality in $n$ is not the generic behavior outside the
standardization-invariant setting.
\end{example}

\section{Computing $d_X(\emptyset;n)$}\label{sec:compute}

The previous section indicates that, particularly in the standardization-invariant case (which is most ubiquitous in application), the crux to computing $d_X(I;n)$ is determining the values $d_X(\emptyset;\cdot)$. This section is devoted to insight on this computation.

\subsection{Hamiltonian Paths}\label{sec:hamps}

 We illustrate that $d_X(\emptyset;n)$ is enumerated by Hamiltonian paths in a certain digraph. This is conceptually useful yet presents a practical challenge: counting Hamiltonian paths is \#P-complete for general digraphs, so one should not expect $d_X(\emptyset;n)$
to be easy for an arbitrary $X \subseteq \mathbb{N}^2$ \cite{garey-johnson,bang-jensen-gutin-digraphs}.

\begin{proposition}\label{prop:recursion}
Let $X\subseteq\N^2$ and $n\ge 1$. Define a directed graph $G_n(X)$ on vertex set $[n]$ by
\[
(i,j)\in E(G_n(X))\quad\Longleftrightarrow\quad (i,j)\notin X,\ \ i\neq j.
\]
Then $d_X(\emptyset;n)$ equals the number of Hamiltonian paths in $G_n(X)$.
\end{proposition}

\begin{proof}
By the construction of \(G_n(X)\) we have
\[
 (\pi_i,\pi_{i+1})\in E(G_n(X))
 \Longleftrightarrow\quad
 \pi_i\neq \pi_{i+1}\ \text{ and }\ (\pi_i,\pi_{i+1})\notin X.
\]
Therefore \(\pi\in\mathcal{D}_X(\emptyset;n)\) if and only if
\(\pi_1\to\pi_2\to\cdots\to\pi_n\) is a directed path in \(G_n(X)\). Since $\pi$ is a permutation of $[n]$, this establishes a bijection between \(\mathcal{D}_X(\emptyset;n)\) and the set of Hamiltonian paths in \(G_n(X)\), and the result follows.
\end{proof}

\medskip

Despite Hamiltonian paths being difficult to count in general, there are effective formulas for certain classes of digraphs. For instance, Grinberg and Stanley introduced the R\'edei-Berge symmetric function $U_D$ attached to a directed graph $D$, whose specializations count Hamiltonian paths and related objects \cite{grinberg-stanley-redei-berge}.
In tournaments, this yields a striking positive formula in terms of permutations and their cycle structures, which we
record next.

\begin{theorem}[Grinberg-Stanley]\label{thm:tournament}
Let $D$ be a tournament on vertex set $[n]$, and let $\overline{D}$ denote its complement (i.e.\ all edges reversed).
For $w\in\Sset_n$, let $\mathrm{nsc}(w)$ be the number of nontrivial cycles in the disjoint cycle decomposition of $w$.
Then the number of Hamiltonian paths in $D$ equals
\[
\sum_{w} 2^{\mathrm{nsc}(w)},
\]
where the sum ranges over all permutations $w\in\Sset_n$ of odd order (equivalently, all cycle lengths of $w$ are odd)
such that every nontrivial cycle of $w$ is a directed cycle of $\overline{D}$.
\end{theorem}
Theorem~\ref{thm:tournament} immediately implies two classical facts about tournaments.
First, every tournament has a Hamiltonian path: the identity permutation is always admissible in the sum (it has no
nontrivial cycles), so the right-hand side is at least $1$.
Second, the number of Hamiltonian paths in a tournament is always odd: the identity contributes $1$, while every other
admissible permutation has $\mathrm{nsc}(w)\ge 1$ and therefore contributes an even number to the sum.
These statements are usually attributed to R\'edei (existence) and Berge (oddness); see \cite{moon-tournaments} and the
historical discussion in \cite{grinberg-stanley-redei-berge}. Grinberg and Stanley's result applies directly to our count as follows:

\begin{theorem}\label{thm:hampath}
Suppose that for each $n$ the directed graph $G_n(X)$ is a tournament.
Then
\[
d_X(\emptyset;n)
=
\sum_{w} 2^{\mathrm{nsc}(w)},
\]
where the sum ranges over all $w\in\Sset_n$ of odd order such that every nontrivial cycle of $w$ is a directed cycle of
$\overline{G_n(X)}$.
\end{theorem}

\begin{proof}
By Proposition~\ref{prop:recursion}, the number \(d_X(\emptyset;n)\) is the number of Hamiltonian paths
in the directed graph \(G_n(X)\).  Now apply Theorem~\ref{thm:tournament} with $D=G_n(X)$ and the result follows.
\end{proof}

\begin{example}
Let $X=\set{(i,j)\in\N^2:\ i<j}$. Then $G_n(X)$ has an edge $i\to j$ exactly when $i>j$, so $G_n(X)$ is the transitive
tournament and has a unique Hamiltonian path $n\to(n-1)\to\cdots\to 1$. Hence by Proposition~\ref{prop:recursion}, $d_X(\emptyset;n)=1$ for all $n$. Theorem~\ref{thm:hampath} predicts this: $\overline{G_n(X)}$ has no nontrivial directed cycles so the only summand in $ \sum_w 2^{\text{nsc}(w)}$ is the contribution of the identity in $e \in \mathfrak{S}_n$, which has no nontrivial cycles in its disjoint cycle decomposition so $\text{nsc}(e)=0$.
\end{example}

\begin{corollary}\label{cor:tournamentpoly}
If $G_n(X)$ is a tournament for each $n$, then $d_X(\emptyset;n)$ is odd for every $n$.
If, in addition, $G(X)\coloneqq\bigcup_{n\ge 1}G_n(X)$ has no directed $3$-cycle, then for all $n$ and all
$I\subseteq[n-1]$ we have $d_X(I;n)=d(I;n)$.
\end{corollary}

\begin{proof}
Assume first that \(G_n(X)\) is a tournament for each \(n\). Berge's classical result \cite{moon-tournaments,grinberg-stanley-redei-berge} tell us that $G_n(X)$ has an odd number of Hamiltonian paths, so the result follows by Proposition~\ref{prop:recursion}.

Now assume in addition that the infinite digraph $G(X)=\bigcup_{n\ge 1}G_n(X)$ has no directed $3$-cycle. Define a relation $\prec$ on $\mathbb{N}$ given by $i \prec j$ if and only if $i \to j$. We claim $\prec$ is a strict total order. Indeed since $G_n(X)$ is a tournament for $n>\max(i,j)$, one of $i \prec j$ or $j \prec i$ holds. It only remains to prove transitivity. Suppose $i \prec j$ and $j \prec k$. Then $i \to j$ and $j \to k$. The absence of $3$-cycles in $G(X)$ omits $k \to i$ as a possibility, so $i \to k$, and hence $i \prec k$. Now fix $n$. List the elements of $\{1,2,\ldots,n\}$ in increasing order $v_1 \prec v_2 \prec \ldots \prec v_n$ with respect to $\prec$. Define the map $\rho:\{1,2,\ldots,n\} \to \{1,2,\ldots,n\}$ given by $\rho(v_j)=j$, so $i \prec j$ if and only if $\rho(i)<\rho(j)$. Now given a permutation $\pi = \pi_1\pi_2 \cdots \pi_n \in \mathfrak{S}_n$, relabel each entry according to its $\rho$ value: $\pi_{\rho} = \rho(\pi_1)\rho(\pi_2) \cdots \rho(\pi_n)$. Then we have
\[
i \in \XD(\pi) \Longleftrightarrow (\pi_i,\pi_{i+1}) \in X \Longleftrightarrow \pi_{i+1} \prec \pi_i \Longleftrightarrow \rho(\pi_{i+1}) < \rho(\pi_i) \Longleftrightarrow i \in \text{Des}(\pi_{\rho}).
\]
So, the bijection on $\mathfrak{S}_n$ sending $\pi \to \pi_{\rho}$ for any $\pi$ induces a bijection $\mathcal{D}_X(I;n) \to \mathcal{D}(I;n)$ and hence $d_X(I;n)=d(I;n)$.
\end{proof}

We close with a general expression for the number of Hamiltonian paths in an arbitrary digraph, due to
Grinberg-Stanley \cite[Thm.~6.6]{grinberg-stanley-redei-berge}. An augmented version of it together with Proposition~\ref{prop:recursion} yields an explicit, albeit computationally expensive, formula for $d_X(\emptyset;n)$.

For set up, let $D=(V,A)$ be a directed graph on a finite vertex set $V$. Write $\overline{D}=(V,\overline{A})$ for its loopless complement, i.e.\ $\overline{A}=V^2\setminus(A\cup\Delta)$ where
$\Delta=\{(v,v):v\in V\}$. A cyclic list $(v_1,\dots,v_k)$ of distinct vertices is called a \emph{$D$-cycle} if $(v_i,v_{i+1})\in A$ for all $i$
(indices mod $k$). Let $\Sset_V$ be the permutations of $V$. If $(\sigma_{i_1} \cdots \sigma_{i_k})$ is a cycle in the disjoint cycle decomposition of a permutation $\sigma \in \mathfrak{S}_V$ then we say it is a $D$-cycle if the cyclic list $(\sigma_{i_1}, \cdots \sigma_{i_k},\sigma_{i_1})$ is. Define
\[
\widetilde{\Sset}_V(D,\overline{D})\coloneqq\set{\sigma\in\Sset_V:\ \text{every non-trivial cycle of $\sigma$ is a $D$-cycle or a $\overline{D}$-cycle}}.
\]
For $\sigma\in\Sset_V$, set
\[
\varphi_D(\sigma)\coloneqq \sum_{\substack{\gamma\in\mathrm{Cycs}(\sigma)\\ \gamma\text{ is a $\overline{D}$-cycle}}} (\abs{\gamma}-1),
\]
where $\mathrm{Cycs}(\sigma)$ is the set of nontrivial cycles in the disjoint cycle decomposition of $\sigma$. The following is an immediate reformulation of Grinberg and Stanley's formula for the number of Hamiltonian paths in $D$:

\begin{theorem}[Grinberg-Stanley]\label{thm:darij}
Let $D=(V,A)$ be a directed graph on a finite vertex set $V$. Then the number of Hamiltonian paths in $D$ is
\[
\sum_{\sigma\in \widetilde{\Sset}_V(D,\overline{D})} (-1)^{\varphi_D(\sigma)}.
\] \qed
\end{theorem}
It immediately follows that $d_X(\emptyset;n)$ can be computed as follows.
\begin{corollary}
\[
d_X(\emptyset;n) = \sum_{\sigma\in \widetilde{\Sset}_V(G_n(X),\overline{G_n(X)})} (-1)^{\varphi_{G_n(X)}(\sigma)}.
\]
\end{corollary}

The general formulas here experience combinatorial explosion when attempting to compute with them for general sets $X$. Our focus in the rest of this section is to compute $d_X(\emptyset;n)$ in cases when $X$ affords the computation to be more tractable. We start with periodic sets $X$, which allow for transfer matrix methods to compute $d_X(\emptyset;n)$.

\subsection{Periodic relations and transfer matrices.}
We expand the theory developed in Theorem~\ref{thm:periodic-reduction} to compute the numbers $d_X(\emptyset;n)$ by an explicit transfer-matrix encoding. The crux of this lies in the following theorem.

\begin{theorem}\label{thm:quasipolynomial}
Let $m\ge 1$ and suppose $X$ is periodic modulo $m$ with indicator $f:(\Z/m\Z)^2\to\{0,1\}$.
Let $H$ be the directed graph on vertex set $\Z/m\Z$ with an edge $r\to s$ iff $f(r,s)=0$.
For $n\ge 1$ write $\ell_r(n)=\abs{\{t\in[n]:t\equiv r\pmod m\}}$.
Let $A_H(\ell_0,\dots,\ell_{m-1})$ be the number of words $w=w_1\cdots w_n$ over $\Z/m\Z$ such that $\abs{\{i:w_i=r\}}=\ell_r(n)$ for all $r$ and $w_i\to w_{i+1}$ is an edge of $H$ for all $i$.
Then
\begin{equation}\label{eq:HKcore}
d_X(\emptyset;n)=A_H\!\bigl(\ell_0(n),\dots,\ell_{m-1}(n)\bigr)\cdot \prod_{r\in\Z/m\Z}\ell_r(n)!.
\end{equation}
Moreover, the multivariate generating function
\[
F(x;y_0,\dots,y_{m-1})
\coloneqq
\sum_{n\ge 1}\ \sum_{\ell_0+\cdots+\ell_{m-1}=n}
A_H(\ell_0,\dots,\ell_{m-1})\,
x^n\,y_0^{\ell_0}\cdots y_{m-1}^{\ell_{m-1}}
\]
is a rational function.
\end{theorem}

\begin{proof}
The identity~\eqref{eq:HKcore} is the specialization \(I=\emptyset\) of Theorem~\ref{thm:periodic-reduction}: when \(I=\emptyset\), the condition \(f(w_i,w_{i+1})=1\) never occurs, so the
allowed residue transitions are exactly those with \(f(r,s)=0\), i.e.\ the edges of the digraph \(H\).
Thus \(\abs{\mathcal{W}_{f,\emptyset}(n)}\) from Theorem~\ref{thm:periodic-reduction} is precisely
\(A_H(\ell_0(n),\dots,\ell_{m-1}(n))\), and multiplying by \(\prod_r \ell_r(n)!\) gives~\eqref{eq:HKcore}.

For the rationality statement, we write down the standard transfer-matrix encoding with letter-weights. Let \(M(y)\) be the \(m\times m\) matrix indexed by residues \(r,s\in\Z/m\Z\) with entries
\[
 M(y)_{rs}\coloneqq \mathbf{1}_{(r\to s)\in E(H)}\cdot y_s.
\]
If we view \(y_s\) as the weight of writing the next letter \(s\), then a product
\(M(y)_{w_1w_2}M(y)_{w_2w_3}\cdots M(y)_{w_{n-1}w_n}\) contributes the monomial \(y_{w_2}\cdots y_{w_n}\) (and is zero
unless all transitions \(w_i\to w_{i+1}\) are edges of \(H\)). Write $\mathbf{y}^T=(y_0,\ldots,y_{m-1})$. The factor $\mathbf{y}^T$ supplies the weight of the inital letter $w_1$. Summing over all choices of \(w_1\) and \(w_n\) for \(n\ge 1\) we have
\[
 \mathbf{y}^T M(y)^{\,n-1}\mathbf{1}
 =
 \sum_{\substack{w=w_1\cdots w_n\\ w_i\to w_{i+1}\text{ in }H}}
 y_{w_2}\cdots y_{w_n}.
\]
It follows then that
\[
 F(x;y_0,\ldots,y_{m-1})=\sum_{n\ge 1} x^n\,\mathbf{y}^T M(y)^{\,n-1}\mathbf{1}
 =
 x\,\mathbf{y}^T\Bigl(\sum_{t\ge 0} (xM(y))^{t}\Bigr)\mathbf{1}
 =
 x\,\mathbf{y}^T(I-xM(y))^{-1}\mathbf{1}.
\]
Since \((I-xM(y))^{-1}\) has entries that are rational functions of \(x\) and \(y_0,\dots,y_{m-1}\), so does \(F\).
Finally, extracting the coefficient of \(y_0^{\ell_0}\cdots y_{m-1}^{\ell_{m-1}}\) recovers
\(A_H(\ell_0,\dots,\ell_{m-1})\) by construction, completing the proof.
\end{proof}

For fixed $m$, the quantity $A_H(\ell_0,\dots,\ell_{m-1})$ can be computed by a simple recursion over the content
vector.  For $r\in\Z/m\Z$ and a vector $\ell=(\ell_0,\dots,\ell_{m-1})$ with $\ell_r\ge 1$, let $e_r$ be the $r$th unit
vector and define
\[
A_H(\ell;r)\coloneqq \#\{\text{admissible words with content }\ell\text{ that end in }r\}.
\]
Then
\[
A_H(\ell;r)=\sum_{s:\,s\to r\text{ in }H} A_H(\ell-e_r;s),
\]
with base cases $A_H(e_r;r)=1$ and $A_H(\ell;r)=0$ if any coordinate of $\ell$ is negative.
Finally $A_H(\ell)=\sum_r A_H(\ell;r)$.  This is the usual transfer-matrix recursion specialized to fixed content
\cite{flajolet-sedgewick,stanley-ec1}.  When $m$ is fixed and $\ell_0+\cdots+\ell_{m-1}=n$, the number of states is on
the order of $n^{m-1}$, so this yields a polynomial-time algorithm in $n$ (for fixed $m$), albeit with exponent
depending on $m$.

We make a note about the generating function 
\[
F(x;y_0,\ldots,y_{m-1})=x\,\mathbf{y}^\top(I-xM(y))^{-1}\mathbf{1}.
\]
This is rational in $x$ and the variables $y_0,\dots,y_{m-1}$, so its coefficients in the multivariate Taylor expansion
satisfy strong algebraic-differential constraints.  A particularly useful general fact is that the diagonal of a
rational power series is $D$-finite (equivalently, its coefficients are \emph{$P$-recursive}).  To connect this to $d_X(\emptyset;n)$, note that in~\eqref{eq:HKcore} we do \emph{not} take an arbitrary coefficient
of $F$, but the coefficient corresponding to the specific content vector
\[
\bigl(\ell_0(n),\dots,\ell_{m-1}(n)\bigr),
\qquad
\ell_r(n)=\abs{\{t\in[n]:t\equiv r\pmod m\}},
\]
which varies with $n$ in a periodic way.
Such extractions can still be expressed using generalized diagonals and root-of-unity filters,
a standard trick in multivariate generating function theory and analytic combinatorics in several variables
\cite{flajolet-sedgewick,pemantle-wilson}.  In particular, for fixed $m$ the sequence
$n\mapsto A_H(\ell_0(n),\dots,\ell_{m-1}(n))$ is $P$-recursive, hence so is $d_X(\emptyset;n)$ after dividing out the
explicit factorial product.

\begin{example}\label{ex:alternating-parity}
Let $m=2$ and define $X$ by $(a,b)\in X$ iff $a\equiv b\pmod 2$.
Then $\XD(\pi)=\emptyset$ means the allowed transitions in $H$ are $0\to 1$ and $1\to 0$ only. The corresponding word count is
$A_H(\ell_0,\ell_1)=0$ unless $|\ell_0-\ell_1|\le 1$; for the specific content vector coming from $[n]$ this condition
always holds, and we get $A_H=2$ when $n$ is even and $A_H=1$ when $n$ is odd.
Therefore
\[
d_X(\emptyset;n)=
\begin{cases}
2\left(\frac{n}{2}!\right)^2 & n \text{ even},\\[4pt]
\left(\frac{n-1}{2}\right)! \cdot \left(\frac{n+1}{2}\right)! & n \text{ odd}.
\end{cases}
\]
This is another instance where a simple periodic $X$ yields a compact closed form for $d_X(\emptyset;n)$.
\end{example}
\begin{example}Let $m=3$ and define $X$ by
\[
(a,b)\in X\quad\Longleftrightarrow\quad b\equiv a+1\pmod 3.
\]
Thus an $X$-descent occurs exactly when consecutive entries advance by one step along the directed $3$-cycle
$0\to 1\to 2\to 0$ in $H$.
For $I=\emptyset$, Theorem~\ref{thm:quasipolynomial} says that $d_X(\emptyset;n)$ is obtained by counting residue
words with the same content as the multiset of residues of $\{1,\dots,n\}$, and avoiding the three forbidden transitions
$0\to 1$, $1\to 2$, and $2\to 0$, then multiplying by the factorial product $\ell_0(n)!\ell_1(n)!\ell_2(n)!$.

The resulting residue-word count is quite nontrivial. One can directly compute the values below:
\[
\begin{array}{c|c|c|c}
n & (\ell_0(n),\ell_1(n),\ell_2(n)) & A_H(\ell_0,\ell_1,\ell_2) & d_X(\emptyset;n)\\ \hline
1 & (0,1,0) & 1 & 1 \\
2 & (0,1,1) & 1 & 1 \\
3 & (1,1,1) & 3 & 3 \\
4 & (1,2,1) & 4 & 8 \\
5 & (1,2,2) & 6 & 24 \\
6 & (2,2,2) & 12 & 96 \\
7 & (2,3,2) & 19 & 456 \\
8 & (2,3,3) & 33 & 2376 \\
9 & (3,3,3) & 66 & 14256 \\
10 & (3,4,3) & 111 & 95904 \\
\end{array}
\]
Here $\ell_r(n)=\abs{\{t\in[n]:t\equiv r\pmod 3\}}$ and $H$ is the digraph obtained by deleting the edges of the directed
$3$-cycle. This illustrates two general features of periodic relations: the ``hard part'' is the constrained word count, while the
factorial product simply accounts for permuting labels within each residue class.
\end{example}

\subsection{Successions.}
Certain classes of sets $X$ give rise to classical techniques for computing $d_X(\emptyset;n)$. For instance, let $X=\set{(a,a+1):a\ge 1}$.
Then $\XD(\pi)=\emptyset$ means that $\pi$ has no \emph{successions}.
This family is classical and has a clean inclusion-exclusion enumeration.

\begin{theorem}\label{thm:delta-egf-corrected}
Let $X=\set{(a,a+1):a\ge 1}$. Then for every $n\ge 1$,
\begin{equation}\label{eq:succ-IE}
d_X(\emptyset;n)=\sum_{k=0}^{n-1}(-1)^k \binom{n-1}{k}\,(n-k)!.
\end{equation}
\end{theorem}

\begin{proof}
For \(\pi=\pi_1\cdots\pi_n\in\Sset_n\), an \(X\)-descent is exactly a \emph{succession} \(\pi_{i+1}=\pi_i+1\).
For each \(a\in[n-1]\), let \(E_a\) be the event that \(a\) is immediately followed by \(a+1\) somewhere in \(\pi\).
Then \(d_X(\emptyset;n)\) counts permutations in which none of the events \(E_a\) occur, so inclusion-exclusion gives
\[
d_X(\emptyset;n)=\sum_{T\subseteq[n-1]} (-1)^{\abs{T}}\,
\#\{\pi\in\Sset_n:\ E_a\text{ holds for all }a\in T\}.
\]
Now it is standard to see that if $|T|=k$, the summand indexed by $T$ is $(-1)^k(n-k)!$ and the result follows.
\end{proof}

\begin{example}Let $X=\{(a{+}1,a):a\ge 1\}$, so that $\XD(\pi)=\emptyset$ means that $\pi$ has no \emph{reverse successions}
(adjacent decreasing consecutive integers).
The map $\pi\mapsto (n{+}1-\pi_1)\cdots(n{+}1-\pi_n)$ is a bijection of $\Sset_n$ that swaps successions and reverse
successions, so the numbers of permutations avoiding successions and avoiding reverse successions coincide.
Hence the closed form and recurrence of Theorem~\ref{thm:delta-egf-corrected} applies.
\end{example}

Forbidding adjacencies of the form $(a,a+d)$ for $d$ in a finite set $\Delta$ leads to a rich family of problems.
As we saw when $\Delta=\{1\}$, inclusion-exclusion reduces this 
to the simple closed form~\eqref{eq:succ-IE}.  For larger $\Delta$, the forbidden adjacency graph becomes a
bounded-degree graph on $\{1,\dots,n\}$ whose connected components, called ``clusters'', can overlap in complicated ways.
In these settings, the Goulden-Jackson cluster method gives a systematic way to express the exponential generating
function in terms of cluster enumerators \cite{goulden-jackson,flajolet-sedgewick}; see also \cite{bona-permutations}
for many related consecutive-pattern computations.

\subsection{A probabilistic lower bound.}
Looking more generally, one wonders what the typical behaviour of $d_X(\emptyset,n)$ is. In that light, let $p\in(0,1)$ and build a random set $X\subseteq[n]^2$ by including each ordered pair $(i,j)$ with $i\neq j$ in
$X$ independently with probability $1-p$. Equivalently, $G_n(X)$ is the random directed graph $\vec{G}(n,p)$ where each
directed edge appears independently with probability $p$ (no loops). Then $d_X(\emptyset;n)$ is the random variable
counting Hamiltonian paths in $\vec{G}(n,p)$.

\begin{theorem}\label{thm:probabilistic}
Let $Y=d_X(\emptyset;n)$;  equivalently $Y$ is the number of
Hamiltonian paths in $\vec{G}(n,p)$.
Then for all $n$,
\[
\mathbb{P}\!\left(Y\ge \tfrac12\,n!\,p^{\,n-1}\right)\ge \tfrac14\,\exp(1-1/p).
\]
\end{theorem}
Theorem~\ref{thm:probabilistic} shows that in particular if $p \in (0,1)$ then with probability bounded away from $0$ (uniformly in $n$), the proportion of permutations in $\mathfrak{S}_n$ that have no $X$-descent set is at least
$\tfrac12\,p^{\,n-1}$. This is far more than polynomially many permutations in $n$.

\begin{proof}
Write \(\vec{G}(n,p)=G_n(X)\), and for \(\pi\in\Sset_n\) let \(\mathbf{1}_\pi\) be the indicator of the event that all
directed edges \((\pi_i,\pi_{i+1})\) for \(1\le i\le n-1\) appear in \(\vec{G}(n,p)\).
Then
\[
Y=\sum_{\pi\in\Sset_n}\mathbf{1}_\pi.
\]
For a fixed \(\pi\), these \(n-1\) edges are distinct and occur independently with probability \(p\), so
\(\mathbb{E}[\mathbf{1}_\pi]=p^{\,n-1}\), and hence \(\mathbb{E}[Y]=n!\,p^{\,n-1}\). Now let \(\pi,\sigma\in\Sset_n\) and write \(E(\pi)=\{(\pi_i,\pi_{i+1}):1\le i\le n-1\}\) for the set
of directed edges used by \(\pi\).
If \(Z=|E(\pi)\cap E(\sigma)|\) is the number of common edges, then
\[
|E(\pi)\cup E(\sigma)|=2(n-1)-Z,
\]
so by independence of edges in \(\vec{G}(n,p)\),
\[
\mathbb{P}(\mathbf{1}_\pi=\mathbf{1}_\sigma=1)=p^{\,2(n-1)-Z}=p^{\,2(n-1)}\,p^{-Z}.
\]
Averaging over \(\pi\) and \(\sigma\) gives
\begin{equation}\label{eq:second-moment-factor}
\mathbb{E}[Y^2]=(n!)^2\,p^{\,2(n-1)}\,\mathbb{E}\!\left[p^{-Z}\right]=\mathbb{E}[Y]^2\,\mathbb{E}\!\left[p^{-Z}\right],
\end{equation}
where on the right \(Z\) is computed for two independent uniform permutations.

We now bound \(\mathbb{E}[p^{-Z}]\) uniformly in \(n\).
Set \(t\coloneqq 1/p>1\).  Since \(t^Z=(1+(t-1))^Z\), we have the binomial expansion
\[
t^Z=\sum_{k=0}^{Z}\binom{Z}{k}(t-1)^k,
\qquad\text{hence}\qquad
\mathbb{E}[t^Z]=\sum_{k\ge 0}(t-1)^k\,\mathbb{E}\!\left[\binom{Z}{k}\right].
\]
To  compute \(\mathbb{E}[\binom{Z}{k}]\), condition on \(\pi\). Let $S \subseteq E(\pi)$ have size $k$.
Then for a uniform random \(\sigma\in\Sset_n\), it is immediate that
\(\mathbb{P}(S\subseteq E(\sigma))=(n-k)!/n!\). Since the set \(E(\pi)\) has size \(n-1\), this implies
\[
\mathbb{E}\!\left[\binom{Z}{k}\,\middle|\,\pi\right]
=
\binom{n-1}{k}\,\frac{(n-k)!}{n!}
=
\frac{n-k}{n}\cdot\frac{1}{k!},
\]
and the right-hand side does not depend on \(\pi\), so the same formula holds unconditionally.
It follows that
\[
\mathbb{E}[p^{-Z}]=\mathbb{E}[t^Z]
=\sum_{k=0}^{n-1}(t-1)^k\,\frac{n-k}{n}\cdot\frac{1}{k!}
\le \sum_{k\ge 0}\frac{(t-1)^k}{k!}
=\exp(t-1)=\exp(1/p-1).
\]
Combining this with~\eqref{eq:second-moment-factor} yields \(\mathbb{E}[Y^2]\le \exp(1/p-1)\,\mathbb{E}[Y]^2\).
The Paley-Zygmund inequality now gives
\[
\mathbb{P}\!\left(Y\ge \tfrac12\,\mathbb{E}[Y]\right)
\ge
\frac{(1/2)^2\,\mathbb{E}[Y]^2}{\mathbb{E}[Y^2]}
\ge \tfrac14\,\exp(1-1/p),
\]
and since \(\mathbb{E}[Y]=n!\,p^{\,n-1}\) the result follows.
\end{proof}

\section{Open Problems}\label{sec:open}

We conclude with list of directions suggested by our study.


\begin{question}Find broad, verifiable conditions on $X$ under which, for each fixed $I$, the sequence $n\mapsto d_X(I;n)$ is
eventually polynomial or quasipolynomial.
Periodic relations yield $P$-recursive sequences via rational generating functions as seen through Theorem~\ref{thm:quasipolynomial} and the discussion afterwards. Are there intermediate classes (e.g.\ relations defined by inequalities on a fixed number of base-$m$ digits, or by
finite automata on the binary expansions of labels) for which one can make conclusions?
\end{question}

\begin{question}For periodic $X$, the core difficulty is the constrained word count 
\[
A_H(\ell_0(n),\dots,\ell_{m-1}(n)).
\]
Can one obtain effective asymptotics for $d_X(\emptyset;n)$ (or more generally for $d_X(I;n)$) in terms of the
spectral data of the residue digraph $H$?
Analytic combinatorics in several variables provides general methods for asymptotics of diagonals and coefficient
extractions from rational functions \cite{pemantle-wilson,flajolet-sedgewick}.
Can these methods be made concrete for natural families of residue digraphs arising from arithmetic relations?
\end{question}

\begin{question} When $G_n(X)$ is a tournament for each $n$, $d_X(\emptyset;n)$ is always odd
(Corollary~\ref{cor:tournamentpoly}).  Which infinite tournaments $G(X)$ force stronger restrictions on
$d_X(\emptyset;n)$?
For example, if $G(X)$ is transitive then $d_X(I;n)$ reduces to ordinary descents (Corollary~\ref{cor:tournamentpoly}).
More generally, if $G(X)$ has bounded ``feedback'' structure (for instance, few directed cycles of certain types), can
one bound or describe the growth of $d_X(\emptyset;n)$?
How do such questions relate to classical extremal problems on tournaments \cite{moon-tournaments}?
\end{question}

\begin{question}Theorem~\ref{thm:probabilistic} gives a uniform positive lower bound on the probability that
$d_X(\emptyset;n)$ is within a constant factor of its expectation in the $\vec{G}(n,p)$ model.
Can one prove sharper concentration for $d_X(\emptyset;n)$, or determine the threshold behavior of
$\mathbb{P}(d_X(\emptyset;n)>0)$ as a function of $p=p(n)$?
These are natural analogues of Hamiltonicity and path-existence thresholds in random graphs
\cite{janson-randomgraphs,alon-spencer}, but the directed setting and the emphasis on \emph{paths} (rather than cycles)
changes the details.
\end{question}

\section*{Acknowledgments}
The author thanks the anonymous referees for thoughtful feedback that greatly improved the manuscript. The author is partially supported by research funds from York University, and NSERC Discovery Grant \#RGPIN-2025-06304.

\bigskip

%
%

\end{document}